\documentclass[a4paper,leqno,11pt]{article}

\usepackage{
amsmath,
amssymb,
amsthm,
verbatim,
fancyhdr,
amsfonts,
subfigure,
mathrsfs
}

\usepackage[hidelinks]{hyperref}

\usepackage{tikz}
\usetikzlibrary{decorations.pathreplacing}

\newcommand\dx{{\,dx}}

\newcommand\G{\mathcal{G}}
\newcommand\K{\mathcal{K}}
\newcommand\B{\mathcal{B}}
\newcommand\R{\mathbb{R}}
\newcommand\C{\mathbb{C}}
\newcommand\N{\mathbb{N}}

\newcommand\vv{\mathrm{v}}

\newtheorem{theorem}{Theorem}[section]
\newtheorem{proposition}[theorem]{Proposition}
\newtheorem{lemma}[theorem]{Lemma}
\newtheorem{corollary}[theorem]{Corollary}

\theoremstyle{remark}
\newtheorem{remark}[theorem]{Remark}

\theoremstyle{definition}
\newtheorem{definition}[theorem]{Definition}
\newtheorem{example}[theorem]{Example}

\tikzstyle{infinito}=[circle,inner sep=0pt,minimum size=0mm]
\tikzstyle{nodo}=[circle,draw,fill,inner sep=0pt,minimum size=0.5*width("k")]

\date{}

\title{On the lack of bound states for certain \\ NLS equations  on metric graphs}

\author{Enrico Serra$^\dagger$, Lorenzo Tentarelli$^\ddagger$
\\ \ \\
{\small  $^\dagger$Dipartimento di Scienze Matematiche ``G.L. Lagrange'' } \\
{\small Politecnico di Torino } \\
{\small Corso Duca degli Abruzzi, 24, 10129 Torino, Italy} \\
{\small \texttt{enrico.serra@polito.it}}\\ \ \\
{\small $^\ddagger$Dipartimento di Matematica e Applicazioni ``R. Caccioppoli'' } \\
{\small Universit\`a degli Studi di Napoli ``Federico II'' } \\
{\small Via Cintia, Monte S. Angelo, I-80126 Napoli, Italy} \\
{\small \texttt{lorenzo.tentarelli@unina.it}}
}

\begin{document}

\maketitle

\begin{abstract}
The purpose of this paper is to prove some results on the absence of bound states for certain nonlinear Schr\"odinger equations on noncompact metric graphs with localized nonlinearity. In particular, we show how the topological and metric properties of graphs affect the existence/nonexistence of bound states. This work completes the discussion initiated in \cite{ST,T}.
\end{abstract}

\smallskip


\section{Introduction}

The study of NLS equations on noncompact metric graphs has gained popularity in the last few years, because (among other things) these equations are expected to describe  the dynamics of \emph{Bose--Einstein condensates} in ramified traps (see \cite{gnusmi,kottos}). In particular, many studies concentrate on a specific NLSE, the cubic focusing \emph{Gross--Pitaevskii} equation,
\begin{equation}
 \label{GP}
 i\psi_t=-\psi_{xx}-|\psi|^2\psi
\end{equation}
on a graph $\G$, with homogeneous {\em Kirchhoff conditions} at the vertices (see \eqref{kirchhoff}). A central role in this line of research is played by \emph{stationary solutions} of prescribed mass (i.e., $L^2$ norm) $\mu>0$, namely functions of the form
\begin{equation}
 \label{stat}
 \psi(t,x)=e^{i\lambda t}u(x),\quad u:\G\to\C,\quad\lambda\in\R,
\end{equation}
which solve \eqref{GP} when the function $u$ solves the stationary NLS equation 
\[u'' +|u|^2u = \lambda u
\]
(see e.g. \cite{acfn3,ast0,ast1,ast2}). The functions $u$ with these properties are called {\em bound states} of mass $\mu$.

The papers \cite{der, noja} present several interesting motivations for investigating a variant of this problem, characterized by the fact that the nonlinearity affects only a {\em compact} part of the graph. One speaks therefore of problems with {\em localized} nonlinearity. In this case, and for a generic power nonlinearity, the bound states $u$ appearing in \eqref{stat} satisfy the same mass constraint
\begin{equation}
\label{pres}
\int_\G |u|^2\dx = \mu
\end{equation}
but solve (for some $\lambda \in \R$) the stationary NLS equation
\begin{equation}
\label{eq}
u'' + \kappa(x)|u|^{p-2}u = \lambda u
\end{equation}
on each edge of $\G$, still with Kirchhoff boundary conditions. The coefficient $\kappa$ is the characteristic function of the 
part of $\G$ where the nonlinearity is located. Bound states satisfy therefore a  {\em double regime}:  nonlinear in a compact part of $\G$ and linear elsewhere. The exponent $p$ is always assumed to be greater than 2; when $p\in (2,6)$,  the problem is called $L^2$--subcritical (see \cite{cazenave}).

In this work we confine ourselves to Kirchhoff boundary conditions. Many other choices (both in the localized and in the non--localized case) are possible, such as, for instance, the case of $\delta$--like conditions at the vertices. Recent results on this topic are presented in \cite{acfn,acfn2,acfn4}.
\medskip

In this paper we are mainly concerned with problems with localized nonlinearity. Specifically,
we consider a noncompact metric graph $\G$ and we assume that the nonlinearity is localized in the compact core $\K$ of $\G$, namely the subgraph of $\G$ consisting of its bounded edges (see Section \ref{set} for precise statements).

Thus a bound state of mass $\mu$ for the NLS equation on $\G$ with nonlinearity localized on $\K$ is a function $u$ that satisfies the mass constraint \eqref{pres} and solves
equation \eqref{eq}
on each edge of $\G$, with Kirchhoff boundary conditions  at each vertex of $\K$.

Defining $H_\mu^1(\G) = \{ u \in H^1(\G) \; :\; \|u\|_{L^2(\G)} = \mu\}$, it is immediate to recognize (see Section \ref{set}) that bound states correspond to critical points on $H_\mu^1(\G)$ of the energy functional
\[
E(u) = \frac12\int_\G |u'|^2\dx - \frac1p\int_\K |u|^p\dx.
\]
If $u$ happens to be not only a critical point of the functional $E$ but an absolute minimizer, it is called a {\em ground} state.

Existence (and multiplicity) of ground and bound states for the NLS equation with localized nonlinearity has been studied in the papers \cite{ST} and \cite{T}, in dependence of the parameters $\mu$ and $p$. We summarize in the next theorem the main results obtained so far in order to explain our motivations.

\begin{theorem} [\cite{ST}, \cite{T}] 
\label{state}
Let $\G$ be a noncompact metric graph with nonempty compact core.
\begin{itemize}
\item[1.] If $p\in (2,4)$, for every $\mu>0$ there exists a ground state of mass $\mu$.
\item[2.] If $p\in (2,6)$, for every $\mu$ large there exist many bound states of mass $\mu$.
\item[3.] If $p\in [4,6)$, for every $\mu$ large there exists a ground state of mass $\mu$.
\item[4.] If $p\in [4,6)$, for every $\mu$ small there exist no ground states of mass $\mu$.
\end{itemize}

\end{theorem}

The unexpected presence of the threshold $p=4$, discovered in \cite{T}, is a peculiar feature of problems with localized nonlinearity. No analogue
has been found so far for the NLS equation with nonlinearity on the whole of $\G$.

\begin{remark} 
\label{scale} We have stated the preceding result in a slightly simplified form (``$\mu$ small/large'') for the sake of clarity. Actually, as observed in \cite{ast2} and \cite{T}, the problem on $\G$ with mass $\mu$ is equivalent, for every $\theta>0$, to the problem on the homothetic graph $\theta^{\frac{2-p}{6-p}}\G$ with mass $\theta\mu$ via the scaling $u(x) \mapsto \theta^{\frac2{6-p}}u(\theta^{\frac{p-2}{6-p}}x)$. From this it follows that if $\ell$ denotes the total length of the compact core $\K$, the quantity $\ell\mu^{\frac{p-2}{6-p}}$ is {\em scale invariant}. Therefore a precise statement of Theorem \ref{state}, and of all our result below, will involve the quantity $\ell\mu^{\frac{p-2}{6-p}}$ instead of $\mu$.
\end{remark}

It is clear from Theorem \ref{state} that, as far as the solvability of equation \eqref{eq} is concerned, there exists a region of the parameters $p$ and $\mu$ where existence of solutions is not assured. Precisely, when $p \in [4,6)$ and $\mu$ (or more correctly $\ell\mu^{\frac{p-2}{6-p}}$)  is small, Theorem \ref{state} says that no solution can appear as a ground state, but leaves open the possibility to solve the problem through the existence of bound states.

The main purpose of this paper is precisely to analyze what happens in the region of parameters where Theorem \ref{state} does not apply. It turns out that in this case the situation is much more involved, and that the solvability of equation \eqref{eq} depends on the topological properties of the graph $\G$, in sharp contrast with the general results provided by Theorem \ref{state}, that do not depend on the graph at all.

Before we outline our main results, we need to spend a few words on the type of bound states $u$ that a graph may support. These are essentially of two kinds. Indeed, either $u$ vanishes identically on all the half-lines, or $u\not\equiv 0$ on at least one half-line. 
In the former case we speak of solutions supported on $\K$, in the latter of solutions supported on $\G$ (notice that a solution supported on $\G$ may vanish identically on {\em some} half-lines, but not on all of them). We will deal with solutions supported on $\K$ in Section \ref{cyc_pend}. This class of solutions is less interesting: if $u\equiv 0$ on all half-lines, it can be considered as a solution of an NLS equation
with {\em any} nonlinearity $f(u)$ (with $f(0)=0$) outside the compact core and, in a sense, the problem loses its identity.
On the contrary, bound states supported on $\G$ are much more interesting because they live on a noncompact domain  and  are really subject to the double regime imposed by the localized nonlinearity, linear on the half-lines and nonlinear in the compact core.

We are now in a position to describe our main results. We will first deal with a generic graph $\G$ and  prove (Theorem \ref{non_lapos}) that for every $p\in [4,6)$ there exists an (explicit) constant $C^*$ such that if $\ell\mu^{\frac{p-2}{6-p}} < C^*$, there are no bound states of mass $\mu$ with $\lambda\ge 0$. From this it follows that under the same conditions there are no bound states supported on $\G$.

Next we identify a particular class of graphs (trees with at most one pendant, see Definitions \ref{pendant}, \ref{tree}) where the preceding result can be much improved. Indeed for this class of graphs we first show (Theorem \ref{nonex}) that for every $p >2$ and every $\mu>0$, there are no bound states with mass $\mu$ and $\lambda \le 0$. Combining the two results we deduce that whenever $\G$ is in this class and $p\in [4,6)$, the condition $\ell\mu^{\frac{p-2}{6-p}} < C^*$ rules out the existence of {\em any} bound state of mass $\mu$.
On this class of graphs the question arising from Theorem \ref{state} has therefore a complete answer: when $p \in [4,6)$ for ``$\mu$ small'', equation \eqref{eq} has no solutions at all.

These results are complemented in the following way. First we show that whenever the graph is not a tree with one pendant, it is possible 
to construct bound states supported on $\K$ (at least for a dense set of the parameters involved). This shows that Theorem \ref{nonex} cannot be extended to more general classes of graphs. Finally we compare the case of localized nonlinearity to the more common case of the ``everywhere nonlinear'' NLS equation. It turns out that some of the phenomena described in this paper are a specific feature of problems with localized nonlinearity. We exhibit indeed a graph $\G$ that admits no ground state for both problems, that has no bound state supported on $\G$ for small $\mu$ when the nonlinearity is localized on the compact core, but {\em does have} bound states supported on $\G$ for {\em every} mass $\mu$ when the nonlinearity affects the whole graph.
\medskip

The paper is structured as follows. In Section \ref{set} we introduce the precise setting and definitions required to describe the problem. 
The main nonexistence results are in Section \ref{pos}, while in Section \ref{cyc_pend} we show that the nonexistence result on trees cannot be extended to other graphs. Finally, Section \ref{full_non} is devoted to the comparison between problems with localized nonlinearity and everywhere nonlinear equations.

\section{Setting and definitions}
\label{set}

We start by recalling some basic definitions on metric graphs (for more details we refer the reader to \cite{ast1, berkolaiko, Kuchment} and references therein).

In this paper a \emph{metric graph} $\G$ is actually a connected {\em multigraph}, where multiple edges  and self--loops are allowed. Each edge is a finite or half-infinite segment of line and the edges are joined at their endpoints (the vertices of $\G$) according to the topology of the graph (see Figure \ref{figuno}). 

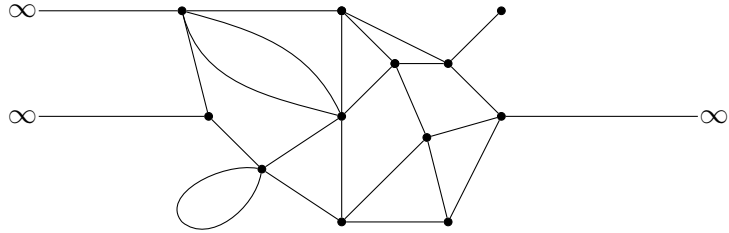
\begin{figure}[ht]
\begin{center}
\begin{tikzpicture}[xscale= 0.7,yscale=0.7]
\node at (-.5,2) [nodo] (02) {};
\node at (2,2) [nodo] (22) {};
\node at (2,4) [nodo] (24) {};
\node at (3.6,1.6) [nodo] (42) {};
\node at (3,3) [nodo] (33) {};
\node at (5,2) [nodo] (52) {};
\node at (3,3) [nodo] (32) {};
\node at (4,3) [nodo] (43) {};
\node at (-1,4) [nodo] (04) {};
\node at (2,4) [nodo] (24) {};
\node at (.5,1) [nodo] (11) {};
\node at (2,0) [nodo] (20) {};
\node at (4,0) [nodo] (40) {};
\node at (5,4) [nodo] (54) {};
\node at (-4,2) [infinito] (meno) {$\infty$};
\node at (-4,4) [infinito] (menoalt) {$\infty$};
\node at (9,2) [infinito] (piu) {$\infty$};

\draw[-] (02)--(04);
\draw[-] (04)--(24);
\draw[-] (04) to [out=-15,in=115] (22);
\draw[-] (04) to [out=-75,in=165] (22);
\draw[-] (24)--(22);
\draw[-] (02)--(11);
\draw[-] (11)--(22);
\draw[-] (11)--(20);
\draw[-] (11) to [out=170,in=130] (-1,0);
\draw[-] (-1,0) to [out=-40,in=260] (11);
\draw[-] (20)--(22);
\draw[-] (22)--(33);
\draw[-] (24)--(33);
\draw[-] (24)--(43);
\draw[-] (33)--(43);
\draw[-] (43)--(52);
\draw[-] (33)--(42);
\draw[-] (20)--(42);
\draw[-] (20)--(40);
\draw[-] (40)--(42);
\draw[-] (42)--(52);
\draw[-] (40)--(52);
\draw[-] (43)--(54);
\draw[-] (02)--(meno);
\draw[-] (52)--(piu);
\draw[-] (04)--(menoalt);
\end{tikzpicture}
\end{center}
\caption{\footnotesize{{a metric graph with 3 half-lines and 22 bounded edges (one pendant).}
 }}
\label{figuno}
\end{figure}

Unbounded edges
are identified with (copies of) $\R^+ = [0,+\infty)$ and are called half-lines, while
bounded edges $e$ are identified with closed bounded intervals $I_e =[0,\ell_e]$, $\ell_e>0$.
In each case a coordinate  $x_e$ is chosen in the corresponding interval, with arbitrary orientation if the interval is bounded, 
and with the natural orientation in case of a half-line.

The graph $\G$ turns in this way into a locally compact metric space, the metric given by the shortest distance along the edges.
Clearly a metric graph is \emph{compact} if and only if it does not contain any half-line. An important role in this paper is played by the following notion, introduced in \cite{ast2,ST}.

\begin{definition}
If $\G$ is a metric graph, we define its \emph{compact core} $\K$ as the metric subgraph of $\G$ consisting of all its bounded edges.
\end{definition}

In what follows, with some abuse of notation, we will say that edges or vertices {\em belong} to the compact core when their points belong to $\K$ as a metric space.

We also denote by $\ell$ the measure of the compact core $\K$, namely
\[
 \ell=\sum_{e\in\K}\ell_e.
\]

\begin{definition}
 \label{pendant}
 We call  \emph{pendant} an edge $e\in\K$ which is incident at a vertex of degree one.
\end{definition}

An example of (noncompact) graph with a pendant is given again by Figure \ref{figuno}. We notice that, by definition, a half-line can never be a pendant.

Finally we recall the following notion.

\begin{definition}
\label{tree}
A \emph{tree} is a graph that contains no \emph{cycles}.
\end{definition}

Trees will play an important role in the next sections.
An example of a noncompact tree is given in Figure \ref{figtre}.

\begin{figure}[ht]
\begin{center}
\begin{tikzpicture}[xscale= 0.5,yscale=0.5]
\node at (0,0) [nodo] (00) {};
\node at (0,2) [nodo] (02) {};
\node at (0,-1) [nodo] (0-1) {};
\node at (3,-1) [nodo] (3-1) {};
\node at (-1,-1) [nodo] (-1-1) {};
\node at (-3,-2) [nodo] (-3-2) {};
\node at (-3,0) [nodo] (-30) {};
\node at (3,3) [nodo] (33) {};
\node at (4,3) [nodo] (43) {};
\node at (3.5,5) [nodo] (355) {};
\node at (2,-2) [nodo] (2-2) {};
\node at (-2,3) [nodo] (-23) {};
\node at (-4,5) [infinito] (-45) {$\infty$};
\node at (8,1) [infinito] (81) {$\infty$};
\node at (7,-5) [infinito] (7-5) {$\infty$};
\node at (2,-6) [infinito] (2-6) {$\infty$};
\node at (-2,-5) [infinito] (-2-5) {$\infty$};
\node at (-8,-2) [infinito] (-8-2) {$\infty$};
\node at (-8,0) [infinito] (-80) {$\infty$};
\draw[-] (00)--(02);
\draw[-] (00)--(0-1);
\draw[-] (00)--(3-1);
\draw[-] (00)--(-1-1);
\draw[-] (-3-2)--(-1-1);
\draw[-] (-30)--(-1-1);
\draw[-] (33)--(02);
\draw[-] (33)--(43);
\draw[-] (33)--(355);
\draw[-] (2-2)--(3-1);
\draw[-] (-23)--(02);
\draw[-] (-45)--(355);
\draw[-] (81)--(43);
\draw[-] (7-5)--(3-1);
\draw[-] (2-2)--(2-6);
\draw[-] (-2-5)--(0-1);
\draw[-] (-3-2)--(-8-2);
\draw[-] (-30)--(-80);
\end{tikzpicture}
\end{center}
\caption{\footnotesize{{a noncompact tree (with one pendant).}
 }}
\label{figtre}
\end{figure}
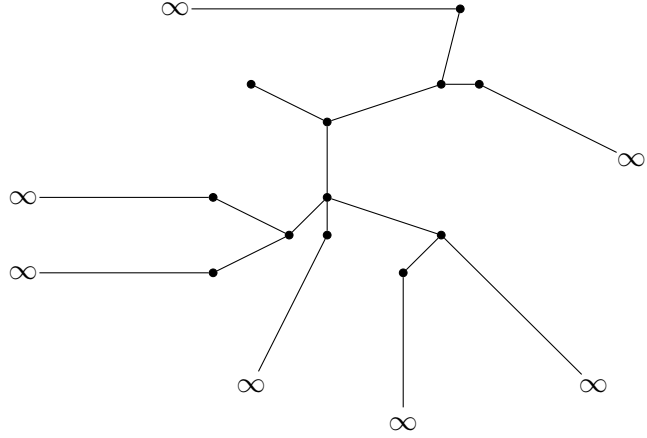

\medskip

A function $u:\G\to\C$ can be regarded as a family of functions $\{u_e\}$, where $u_e:I_e\to\C$  is the restriction of $u$ to the edge (represented by) $I_e$. The usual $L^p$ spaces can be defined over $\G$ in the natural way, with norm
\[
\|u\|_{L^p(\G)}^p = \sum_{e} \| u_e\|_{L^p(I_e)}^p,
\]
while $H^1(\G)$ is the space of continuous $u:\G\to \C$  such that $u_e\in  H^1(I_e;\C)$ for every edge $e$, with norm
\[
 \|u\|_{H^1(\G)}^2 = \|u'\|_{L^2(\G)}^2 + \|u\|_{L^2(\G)}^2.
\]
Continuity at a vertex $\mathrm{v}$ means that different components $u_e$
meeting at $\mathrm{v}$ agree. Further details can be found in \cite {ast1}.

For  $\mu >0$ we set
\[
H_{\mu}^1(\G)= \{u\in H^1(\G)  : \|u\|_{L^2(\G)}^2 = \mu\},
\]
and we define bound states of prescribed mass as follows (see \cite{ST}).

\begin{definition}
\label{defbs}
Let $\G$ be a metric graph with nonempty compact core $\K$ and let $\kappa:\G\to \{0,1\}$ be the characteristic function of $\K$. Finally, let $p > 2$.
We say that a function $u:\G\to\C$ is a \emph{bound state of mass $\mu$ for the NLS equation on $\G$ with nonlinearity localized on $\K$  and Kirchhoff conditions} if:
\begin{equation}
 \label{mass}
 u\in H_{\mu}^1(\G),
\end{equation}
there exists $\lambda\in\R$ such that for every edge $e$
\begin{equation}
\label{NLS}
 u_e'' + \kappa(x)|u_e|^{p-2}\,u_e = \lambda u_e,
\end{equation}
and for every vertex $\vv\in \K$
\begin{equation}
 \label{kirchhoff}
 \sum_{e\succ\mathrm{v}}\frac{du_e}{dx_e}(\mathrm{v}) = 0.
\end{equation}
The symbol $e\succ\mathrm{v}$ means that the sum is extended to all edges $e$ incident at ${\mathrm v}$, while  $\frac{du_e}{dx_e}(\mathrm{v})$ stands for $u_e'(0)$ or $-u_e'(\ell_e)$, according to whether $x_e$ is equal to $0$ or $\ell_e$ at $\mathrm{v}$.
\end{definition}

The final requirement in the preceding definition is called the \emph{Kirchhoff condition}. For simplicity, we refer to a function $u$ satisfying Definition \ref{defbs} as a \emph{bound state of mass} $\mu$.

It is straightforward to check that \eqref{mass}--\eqref{kirchhoff} are equivalent to their weak formulation: $u\in H_\mu^1(\G)$ and
for every $v \in H^1(\G)$
\begin{equation}
 \label{weakbound}
 \mathrm{Re}\left(\int_\G u'\overline{v}'\dx\right) - \mathrm{Re}\left(\int_\K |u|^{p-2}\,u\overline{v} \dx\right) +  \lambda \, \mathrm{Re}\left(\int_\G u\overline{v}\dx\right)=0,
\end{equation}
with $\overline{v}$ denoting the conjugate of $v$ (see \cite{ST}). 

It is thus clear that $u$ is a bound state of mass $\mu$ if and only if it is a critical point of the energy functional 
\begin{equation}
\label{functe}
E(u) = \frac12\|u'\|_{L^2(\G)}^2 - \frac1p\|u\|_{L^p(\K)}^p = \frac12\int_\G |u'|^2\dx - \frac1p\int_\K |u|^p\dx
\end{equation}
over $H_\mu^1(\G)$. The constant $\lambda$ appearing in \eqref{NLS} and \eqref{weakbound} is a {\em Lagrange multiplier}, arising because the funcional $E$ is constrained on the manifold $H_\mu^1(\G)$ (the sphere of $L^2$ of radius $\sqrt\mu$ in $H^1(\G)$).

In particular, {\em ground} states of mass $\mu$ are the absolute minimizers of $E$. 
Up to the multiplication by a constant phase, they are real--valued and of constant sign.

We note, for future reference, that setting $v=u$ in \eqref{weakbound} one finds that
\begin{equation}
\label{lam}
\lambda = \lambda(u) = \frac1\mu \left(\int_\K |u|^p\dx - \int_\G |u'|^2\dx\right).
\end{equation}

The following definition will help in stating our main results. 

\begin{definition}
We say that a bound state $u$ is {\em supported on} $\K$ if $u$ vanishes identically on all the half-lines of $\G$. Otherwise (i.e. if $u$
does not vanish on at least one half-line), we say that it is {\em supported on} $\G$.
\end{definition}

Finally, we mention the graph version of two cases of the Gagliardo--Nirenberg inequality. The proof coincides almost completely with that for  real--valued function that can be found in \cite{T,ast2}. We sketch it briefly for the sake of completeness.

\begin{proposition}
For every $p\in[2,\infty]$ there exists a constant $\mathcal{C}_p>0$ such that
\begin{equation}
\label{GN}
\|u\|_{L^p(\G)}^p \leq \mathcal{C}_p\|u\|_{L^2(\G)}^{\frac{p}{2}+1}\|u'\|_{L^2(\G)}^{\frac{p}{2}-1}\qquad\forall u\in H^1(\G)\qquad\hbox{ if $\; p<\infty$}
\end{equation}
and
\begin{equation}
\label{GNI}
\|u\|_{L^{\infty}(\G)} \leq \mathcal{C}_{\infty}\|u\|_{L^2(\G)}^{1/2}\|u'\|_{L^2(\G)}^{1/2}\qquad\forall u\in H^1(\G) \qquad\hbox{ if $\; p=\infty$}
\end{equation}
for every noncompact metric graph $\G$.
\end{proposition}

\begin{proof}
 Following \cite{ast1}, given a real--valued nonnegative $u\in H^1(\G)$ we can define its \emph{decreasing rearrangement} as the function $u^*:\R^+\to\R$ such that
 \[
  u^*(x)=\inf\{t\geq0:\rho(t)\leq x\},
 \]
 with
 \[
  \rho(t)=\sum_{e}\mathrm{meas}\{x_e\in I_e:u_e(x_e)\geq t\},\qquad t\geq0.
 \]
 One can prove (see again \cite{ast1}) that $u^*\in H^1(\R^+)$, $\sup_{\G}u=\sup_{\R^+}u^*$,
 \[
  \int_{\G}|u|^r\dx=\int_{\R^+}|u^*|^r\dx\qquad\text{and}\qquad\int_{\R^+}|(u^*)'|^2\leq\int_{\G}|u'|^2\dx.
 \]
  Now, let $u\in H^1(\G)$ be a generic complex--valued function. Since
 \[
  ||u|'(x)|\leq|u'(x)| \qquad \hbox{for a.e. $x$ in }\;\G,
 \]
from the classical Gagliardo--Nirenberg inequality in $\R^+$ (\cite{dolbeault}) we obtain
 \begin{align*}
  \|u\|_{L^p(\G)}^p & = \||u|^*\|_{L^p(\R^+)}^p\leq\mathcal{C}_p\||u|^*\|_{L^2(\R^+)}^{\frac{p}{2}+1}\|(|u|^*)'\|_{L^2(\R^+)}^{\frac{p}{2}-1}\\[.2cm]
                    & \leq \mathcal{C}_p\|u\|_{L^2(\G)}^{\frac{p}{2}+1}\||u|'\|_{L^2(\G)}^{\frac{p}{2}-1}\leq \mathcal{C}_p\|u\|_{L^2(\G)}^{\frac{p}{2}+1}\|u'\|_{L^2(\G)}^{\frac{p}{2}-1}.
 \end{align*}
 In the very same way one can prove \eqref{GNI}.
\end{proof}

\section{Nonexistence results}
\label{pos}

In this section we prove the main nonexistence results. In the sequel we tacitly assume that $\G$ is a noncompact metric graph with a nonempty compact core $\K$, where the nonlinearity is located, and that $\mu>0$. We also recall that $\ell$ denotes the mesaure of the compact core. 

First we show that when $p\in [4,6)$ the nonnegativity of the Lagrange multiplier $\lambda$ implies a double estimate on the kinetic part of any bound state in terms of its mass.

\begin{lemma}
\label{twoest}
Let $p\in[4,6)$. Assume that $u$ is a bound state of mass $\mu$ with $\lambda\geq0$. Then
\begin{equation}
\label{alt}
\int_{\G}|u'|^2\dx\leq\mathcal{C}_p^{\frac{4}{6-p}}\mu^{\frac{p+2}{6-p}}
\end{equation}
and
\begin{equation}
\label{bas}
\left(\int_{\G}|u'|^2\dx\right)^{\frac{p-4}{4}}\geq\mathcal{C}_{\infty}^{-p}\ell^{-1}\mu^{-\frac{p}{4}}
\end{equation}
 where $\mathcal{C}_p$ and $\mathcal{C}_{\infty}$ are the constants appearing in inequalities \eqref{GN} and \eqref{GNI}.
\end{lemma}

\begin{proof}
 Since $\lambda\geq0$, by \eqref{lam},
 \begin{equation}
  \label{aux}
  \int_{\G}|u'|^2\dx\leq\int_{\K}|u|^p\dx
 \end{equation}
 and hence, using \eqref{GN},
 \[
  \int_{\G}|u'|^2\dx\le \int_{\K}|u|^p\dx \le   \int_{\G}|u|^p\dx \le \mathcal{C}_p\mu^{\frac{p+2}{4}}\left(\int_{\G}|u'|^2\dx\right)^{\frac{p-2}{4}}.
 \]
 Observing that $u'\not\equiv0$ and then dividing by $\left(\int_{\G}|u'|^2\dx\right)^{\frac{p-2}{4}}$, one obtains
 \[
  \left(\int_{\G}|u'|^2\dx\right)^{\frac{6-p}{4}}\leq\mathcal{C}_p\mu^{\frac{p+2}{4}}
 \]
 and thus \eqref{alt} is proved.
 
Next, from \eqref{aux}, we also see that
 \[
 \int_{\G}|u'|^2\dx\le \int_{\K}|u|^p\dx \le \ell\|u\|_{L^{\infty}(\K)}^p  \le \ell\|u\|_{L^{\infty}(\G)}^p,
 \]
 and using \eqref{GNI}, we find
 \[
  \int_{\G}|u'|^2\dx\leq\mathcal{C}_{\infty}^p\ell\mu^{\frac{p}{4}}\left(\int_{\G}|u'|^2\dx\right)^{\frac{p}{4}}.
 \]
 Hence, as $p\geq4$, dividing by $\int_{\G}|u'|^2\dx$ and suitably rearranging terms, \eqref{bas} follows.
\end{proof}

Now we can prove a first nonexistence result. In its statement it is convenient to keep in mind Remark \ref{scale}.

\begin{theorem}
 \label{non_lapos}
 Let $p\in[4,6)$. Assume that
 \begin{equation}
  \label{mainass}
 \ell \mu^{\frac{p-2}{6-p}}<\mathcal{C}_{\infty}^{-p}\mathcal{C}_p^{\frac{4-p}{6-p}},
 \end{equation}
 where $\mathcal{C}_p$ and $\mathcal{C}_{\infty}$ are again the constants appearing in \eqref{GN} and \eqref{GNI}. Then, there are no bound states of mass $\mu$ with $\lambda\geq0$. In particular, there are no bound states supported on $\G$.
\end{theorem}

\begin{proof}
Assume that there exists a bound state $u$ with $\lambda\geq0$ and mass $\mu$ satisfying \eqref{mainass}. If $p=4$, then \eqref{bas} immediately implies
 \[
  \ell\mu\geq\mathcal{C}_{\infty}^{-4},
 \]
which contradicts \eqref{mainass} when $p=4$.
 
When $p>4$, combining \eqref{bas} and \eqref{alt}, we find that
 \[
  1\leq\ell\mathcal{C}_{\infty}^p\mu^{\frac{p}{4}}\left(\mathcal{C}_p^{\frac{4}{6-p}}\mu^{\frac{p+2}{6-p}}\right)^{\frac{p-4}{4}},
 \]
 whence, with some easy computations,
 \[
 \ell \mu^{\frac{p-2}{6-p}}\geq\mathcal{C}_{\infty}^{-p}\mathcal{C}_p^{\frac{4-p}{6-p}},
 \]
contradicting again \eqref{mainass}. Thus there are no bound states with $\lambda \ge 0$. Note also that if $u$ is a bound state supported on $\G$, then by definition $u\not\equiv 0$ on at least one half-line. Since on half-lines $u'' = \lambda u$ and $u$ is $L^2$, necessarily $\lambda>0$, and this is why this type of bound states is ruled out.
\end{proof}

\begin{remark} Condition \eqref{mainass} could be made sharper by using the specific Gagliardo--Nirenberg constants of the graph $\G$. These however depend on the topology and on the metric properties of the graph (e.g. the lengths of its edges) in an unaccessible way, at least for now. We prefer to use the universal constants of the half-line $\mathcal{C}_p,\,\mathcal{C}_\infty$ for two reasons: first because in this way the inequalities \eqref{GN} and \eqref{GNI} hold for every noncompact graph, and secondly because these constants are explicit (see \cite{dolbeault}). For example, since $\mathcal{C}_\infty = \sqrt 2$, in the model case $p=4$ (and for every graph) condition 
\eqref{mainass} reads simply
\[
\ell \mu < \frac14.
\]
\end{remark}

Before proceeding, we note that assumption \eqref{mainass} is the condition that guarantees in \cite{T} the nonexistence of ground states. Here, however, it is used to prove a far stronger result. In addition, arguing as in the preceding results, one can slightly improve Theorem 3.4 in \cite{T} (raising the nonexistence threshold for ground states), with a completely different (and simpler) proof. 

\begin{corollary}
 Let $p\in[4,6)$. If
 \begin{equation}
  \label{mainassdue}
  \ell\mu^{\frac{p-2}{6-p}}<\left(\frac{p}{2}\right)^{\frac{2}{6-p}}\mathcal{C}_{\infty}^{-p}\mathcal{C}_p^{\frac{4-p}{6-p}},
 \end{equation}
then there is no ground state of mass $\mu$.
\end{corollary}

\begin{proof} By Theorem 3.1 of \cite{T}, for every $\mu>0$,
\begin{equation}
\label{levneg}
\inf_{u\in H_{\mu}^1(\G)} E(u)\leq0.
\end{equation}
If $u\in H_{\mu}^1(\G)$ satisfies  $E(u)\leq0$, by \eqref{functe},
 \[
  \int_{\G}|u'|^2\dx\leq\frac{2}{p}\int_{\K}|u|^p\dx.
 \]
 Now, arguing as in the proof of Lemma \ref{twoest} one obtains that
 \[
  \int_{\G}|u'|^2\dx\leq\left(\frac{2\mathcal{C}_p}{p}\right)^{\frac{4}{6-p}}\mu^{\frac{p+2}{6-p}}
 \]
 and
 \[
  \left(\int_{\G}|u'|^2\dx\right)^{\frac{p-4}{4}}\geq\mathcal{C}_{\infty}^{-p}\left(\frac{2\ell}{p}\right)^{-1}\mu^{-\frac{p}{4}}.
 \]
 Combining the two inequalities, there results
 \[
\ell  \mu^{\frac{p-2}{6-p}}\geq\left(\frac{p}{2}\right)^{\frac{2}{6-p}}\mathcal{C}_{\infty}^{-p}\mathcal{C}_p^{\frac{4-p}{6-p}},
 \]
 which contradicts \eqref{mainassdue}. Summing up, if $\mu$ satisfies \eqref{mainassdue}, then $E(u)>0$ for every $u\in H_{\mu}^1(\G)$. 
 In view of \eqref{levneg}, this  entails that there is no ground state of mass $\mu$.
\end{proof}

For certain classes of graphs Theorem \ref{non_lapos} can be used to prove a full nonexistence result for bound states of small mass, independently of their support. The general class of graphs that enjoys this property is that of {\em trees} with at most one {\em pendant} (Definitions \ref{pendant} and \ref{tree}). Some significative examples of this type of graphs are depicted in Figure \ref{examples}. 

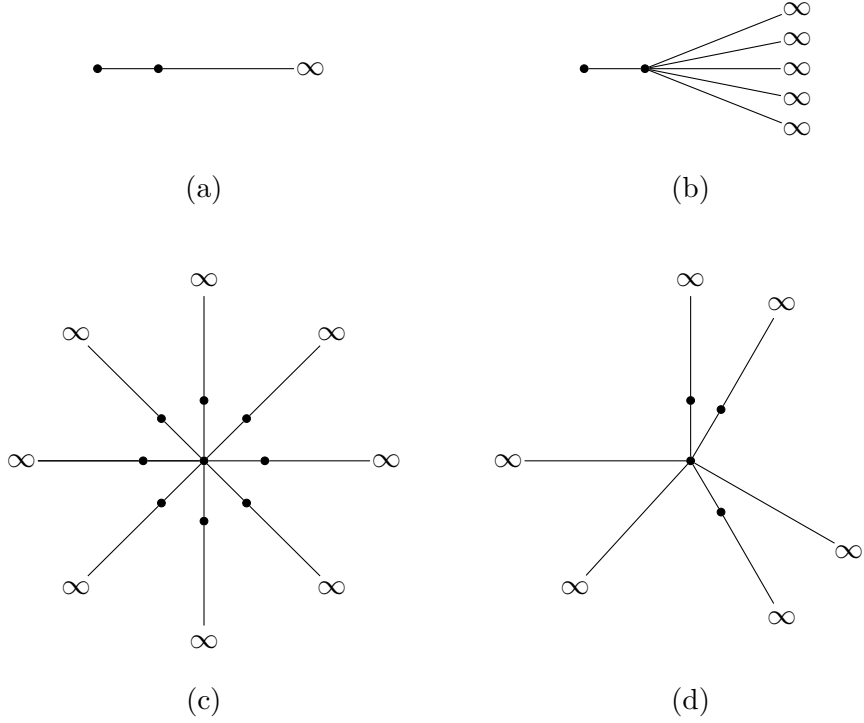
\begin{figure}[th]
\begin{center}
\begin{tikzpicture}[xscale= 0.4,yscale=0.4]
\node at (-8,0) [nodo] (00) {};
\node at (-6,0) [nodo] (20) {};
\node at (-1,0) [infinito] (70) {$\infty$};
\node at (-1,-5.35) [infinito] (fittizio) {};
\draw[-] (00)--(20);
\draw[-] (20)--(70);
\node at (-4.5,-4) [infinito] () {(a)};
\node at (8,0) [nodo] (00) {};
\node at (10,0) [nodo] (20) {};
\node at (15,2) [infinito] (72) {$\infty$};
\node at (15,1) [infinito] (71) {$\infty$};
\node at (15,0) [infinito] (70) {$\infty$};
\node at (15,-1) [infinito] (7-1) {$\infty$};
\node at (15,-2) [infinito] (7-2) {$\infty$};
\draw[-] (00)--(20);
\draw[-] (20)--(72);
\draw[-] (20)--(71);
\draw[-] (20)--(70);
\draw[-] (20)--(7-1);
\draw[-] (20)--(7-2);
\node at (11.5,-4) [infinito] () {(b)};
\node at (-4.5,-13) [nodo] (00) {};
\node at (-2.5,-13) [nodo] (20) {};
\node at (-6.5,-13) [nodo] (-20) {};
\node at (-4.5,-11) [nodo] (02) {};
\node at (-4.5,-15) [nodo] (0-2) {};
\node at (-3.1,-11.6) [nodo] (1414) {};
\node at (-3.1,-14.4) [nodo] (14-14) {};
\node at (-5.9,-11.6) [nodo] (-1414) {};
\node at (-5.9,-14.4) [nodo] (-14-14) {};
\node at (1.5,-13) [infinito] (60) {$\infty$};
\node at (-10.5,-13) [infinito] (-60) {$\infty$};
\node at (-4.5,-7) [infinito] (06) {$\infty$};
\node at (-4.5,-19) [infinito] (0-6) {$\infty$};
\node at (-.3,-8.8) [infinito] (4242) {$\infty$};
\node at (-.3,-17.2) [infinito] (42-42) {$\infty$};
\node at (-8.7,-8.8) [infinito] (-4242) {$\infty$};
\node at (-8.7,-17.2) [infinito] (-42-42) {$\infty$};
\draw[-] (00)--(20);
\draw[-] (00)--(02);
\draw[-] (00)--(-20);
\draw[-] (00)--(0-2);
\draw[-] (00)--(-20);
\draw[-] (00)--(1414);
\draw[-] (00)--(14-14);
\draw[-] (00)--(-1414);
\draw[-] (00)--(-14-14);
\draw[-] (60)--(20);
\draw[-] (06)--(02);
\draw[-] (-60)--(-20);
\draw[-] (0-6)--(0-2);
\draw[-] (-60)--(-20);
\draw[-] (4242)--(1414);
\draw[-] (42-42)--(14-14);
\draw[-] (-4242)--(-1414);
\draw[-] (-42-42)--(-14-14);
\node at (-4.5,-21) [infinito] () {(c)};
\node at (11.5,-13) [nodo] (00) {};
\node at (11.5,-11) [nodo] (02) {};
\node at (11.5,-7) [infinito] (06) {$\infty$};
\node at (5.5,-13) [infinito] (-60) {$\infty$};
\node at (12.5,-11.3) [nodo] (117) {};
\node at (14.5,-7.8) [infinito] (352) {$\infty$};
\node at (12.5,-14.7) [nodo] (1-17) {};
\node at (14.5,-18.2) [infinito] (3-52) {$\infty$};
\node at (16.7,-16) [infinito] (52-3) {$\infty$};
\node at (7.7,-17.2) [infinito] (-42-42) {$\infty$};
\draw[-] (00)--(02);
\draw[-] (06)--(02);
\draw[-] (00)--(-60);
\draw[-] (00)--(117);
\draw[-] (117)--(352);
\draw[-] (00)--(1-17);
\draw[-] (1-17)--(3-52);
\draw[-] (00)--(52-3);
\draw[-] (00)--(-42-42);
\node at (11.5,-21) [infinito] () {(d)};
\end{tikzpicture}
\end{center}

\caption{\footnotesize{topical examples of noncompact trees: (a) segment and half-line; (b) segment and several half-lines; (c) $N$--star graph with nonlinearity affecting a compact portion of each half-line; (d) $N$--star graph with nonlinearity affecting only some half-lines.}}
\label{examples}
\end{figure}

The particular feature of this class is that it is possible to prove {\em a priori} and without any restriction on $p > 2$ and $\mu$ that they do not admit any bound state with Lagrange multiplier $\lambda \le 0$.

\begin{theorem}
\label{nonex}
Let $\G$ be a noncompact tree with at most one pendant. Then for every $p > 2$ and every $\mu>0$ there is no bound state of mass $\mu$ with $\lambda\leq0$.
\end{theorem}

\begin{proof} Assume that $u$ is a bound state of mass $\mu$ with $\lambda \le 0$ ($p > 2$ and $\mu>0$ are understood).
Clearly $u\equiv 0$ on every half-line of $\G$, since there $u''=\lambda u$ with $\lambda \le 0$, and then, if $u$ does not vanish identically, it cannot be in $L^2(\G)$.

We also note the following property: if $\vv$ is a vertex of degree $n\ge 2$ and if $u$ vanishes
identically on $n-1$ edges incident at $\vv$, then it vanishes on all $n$ edges. In order to see this, identify an edge $e$ incident at $\vv$ with $[0,\ell_e]$ (attached at $\vv$ when $x=0$) or with $[0,+\infty)$ and assume that $u\equiv 0$ on all other edges incident at $\vv$. Then, by continuity, $u_e(0) = 0$ and, by the Kirchhoff conditions, $u_e'(0)=0$. Therefore, by the uniqueness of the solution of the Cauchy problem

\[
 \left\{
 \begin{array}{l}
  \displaystyle u_e'' + \kappa(x)|u_e|^{p-2}\,u_e = \lambda u_e\\[.2cm]
  \displaystyle u_e(0) = 0\\[.2cm]
  \displaystyle u_e'(0)=0,
 \end{array}
 \right.
\]
one obtains that $u_e\equiv 0$ on $e$.

Now, consider first trees with no pendants. As $\mu >0$,  $u\not\equiv 0$ on some (necessarily finite) edge $e_1$. Let $\vv_1$ be a vertex of $e_1$. By the preceding property there is at least one edge $e_2\ne e_1$ incident at $\vv_1$ where $u\not\equiv0$. If $e_2$ is a half-line we have reached a contradiction. Otherwise, $e_2$ is a finite edge and hence one can repeat the procedure starting from $e_2$. In this way one can construct a path starting from $e_1$ and consisting of edges where $u\not\equiv0$. Since $\G$ is a tree, it contains no cycles, and therefore the last edge of the path is a half-line, where $u\equiv 0$, and this is a contradiction.

Assume, finally, that the graph has a single pendant $e$. If $u\not\equiv 0$ on $e$, then we set $e_1 =e$ and we repeat the above argument. On the other hand, if $u\equiv0$ on $e$, we remove $e$ from the graph, and we are in the preceding case. In every case we reach a contradiction whenever $u\not \equiv 0$ on some edge.
\end{proof}

\begin{remark}
 The preceding result is valid also in the everywhere nonlinear case (with the same proof). In other words, by Theorem \ref{nonex}, the NLS equation  with  nonlinearity on the whole of $\G$ does not admit solutions for $\lambda\leq0$ on trees with at most one pendant.
\end{remark}

\begin{remark}
\label{false}
If $\G$ is not a tree, or if it is a tree with {\em at least} two pendants, then Theorem \ref{nonex} is false (see Section \ref{cyc_pend}).
\end{remark}

Combining Theorems \ref{non_lapos} and \ref{nonex}, we can prove the following nonexistence result.

\begin{corollary}
Let $\G$ be a tree with at most one pendant. Assume that $p\in[4,6)$ and that 
\begin{equation}
\label{bis}
\mu^{\frac{p-2}{6-p}}\ell<\mathcal{C}_{\infty}^{-p}\mathcal{C}_p^{\frac{4-p}{6-p}}
\end{equation}
Then there is no bound state of mass $\mu$.
\end{corollary}

\begin{proof} Since $\G$ is a tree with at most one pendant, there are no bound states with $\lambda \le 0$ (for any $\mu$) by Theorem  \ref{nonex}. When $\lambda >0$, condition \eqref{bis} excludes bound states via Theorem \ref{non_lapos}.
\end{proof}

\section{Bound states with compact support}
\label{cyc_pend}

The aim of this section is to discuss more in detail the claim of Remark \ref{false}. Precisely, we construct some examples that make clear the reasons why Theorem \ref{nonex} does not apply to graphs containing cycles and/or two or more pendants.

The technique we use below has been formerly introduced for studying the everywhere nonlinear problem on the \emph{tadpole} graph (Figure \ref{tadpole}). In particular, the case $p=4$ was studied in \cite{CFN} while \cite{NPS} deals with the general case $p\in(2,6)$.

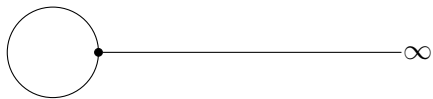
\begin{figure}[ht]
\begin{center}
\begin{tikzpicture}[xscale= 0.6,yscale=0.6]

\node at (0,0) [nodo] (00) {};
\node at (7,0) [infinito] (70) {$\infty$};

\draw [-] (00) -- (70);
\draw (-1,0) circle (1cm);

\end{tikzpicture}
\end{center}
\caption{\footnotesize{a \emph{tadpole} graph.}}
\label{tadpole}
\end{figure}

The obstruction in constructing solutions that vanish identically on the half-lines of the graph, is that one forces the value
of the solution to be zero at the vertices where the half-lines are incident but wants, at the same time, to preserve the Kirchhoff
condition. In order to do this we make use of the following lemma,
which is an immediate consequence of Proposition 2.4 of \cite{NPS}. 

\begin{lemma}
\label{int}
For every $L>0$, $p>2$ and $\lambda\in\R$, there exists an 
odd, L--periodic, smooth function $\phi$ that solves
\[
\phi''+|\phi|^{p-2}\phi=\lambda\phi.
\]
\end{lemma}

Using this lemma we can exhibit a wide class of graphs with cycles that violate Theorem \ref{nonex}.

\begin{theorem}
\label{theocyc}
Let $p>2$ and $\lambda\in\R$. Assume that $\G$ contains a cycle $\mathscr{C}$, whose edges have pairwise commensurable lengths. Then there exists $\mu>0$ for which there is at least a bound state of mass $\mu$, with Lagrange multiplier $\lambda$ and supported on $\K$.
\end{theorem}

\begin{remark}
The rational dependence assumption on the lengths of the edges clearly prevents this result to be valid in generically (see  \cite{BL} for the discussion
of a similar phenomenon in the linear case). However the main feature of the theorem is that it shows that Theorem \ref{nonex} cannot be expected to hold when the graph contains a cycle.

\end{remark}

\begin{proof}[Proof of Theorem \ref{theocyc}]
Let $e_1,\dots,e_n$ denote the consecutive edges of $\mathscr{C}$ (see for instance Figure \ref{figciclo}). 
 \begin{figure}[ht]
 \begin{center}
 \begin{tikzpicture}[xscale= 0.6,yscale=0.6]

 \node at (0,0) [nodo] (00) {};
 \node at (-.5,0) [infinito] () {$\vv$};
 \node at (2,2) [nodo] (22) {};

 \node at (.6,1.3) [infinito] () {$e_1$};
 \node at (5,2.5) [nodo] (525) {};

 \node at (3.4,2.7) [infinito] () {$e_2$};
 \node at (5,.5) [nodo] (55) {};
 
 \node at (5.5,1.5) [infinito] () {$e_3$};
 \node at (7,-1.5) [nodo] (715) {};
 
 \node at (6.5,-.5) [infinito] () {$e_4$};
 \node at (3.5,-2.5) [nodo] (35-25) {};

 \node at (5.5,-2.4) [infinito] () {$e_5$};
 \node at (1.4,-1.5) [infinito] () {$e_6$};
 
 \draw [-] (00) -- (22);
 \draw [-] (525) -- (22);
 \draw [-] (525) -- (55);
 \draw [-] (715) -- (55);
 \draw [-] (715) -- (35-25);
 \draw [-] (00) -- (35-25);
 \draw [dashed] (00) -- (-3,-4);
 \draw [dashed] (22) -- (-3,2);
 \draw [dashed] (525) -- (10,2.5);
 \draw [dashed] (55) -- (2,-.5);
 \draw [dashed] (715) -- (10,-4.5);
 \draw [dashed] (35-25) -- (9,-5.5);
 \draw [dashed] (35-25) -- (-1.5,-3.5);

 \end{tikzpicture}
 \end{center}
 \caption{\footnotesize{an example of a cycle $\mathscr{C}$ consisting of 6 edges}.}
 \label{figciclo}
 \end{figure}
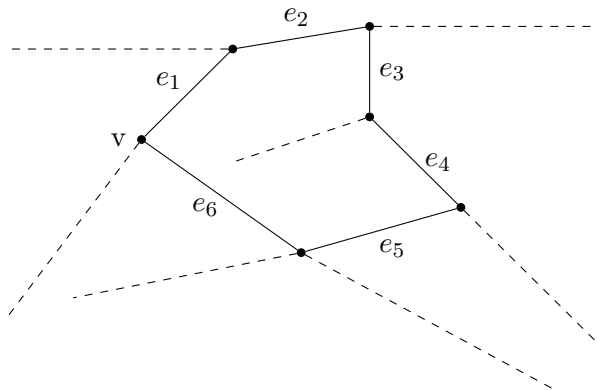
 Since the lengths $\ell_1,\dots,\ell_n$ of these edges are pairwise commensurable,  there exists $L>0$ such that
\[
\ell_i=k_i L,\quad\text{with }k_i\in\N,\quad\forall i\in\{1,\dots,n\}.
\]
Set $k = \sum_{i=1}^n k_i$.  The cycle $\mathscr{C}$ can be identified with the interval $[0,kL]$, the coordinate $x=0$ corresponding to a fixed vertex $\vv$. Clearly, at each vertex the coordinate $x$ (read in $[0,kL]$) is an integer multiple of $L$.

Now, for fixed (but arbitrary) $\lambda\in\R$ and $p>2$, let $\phi$ be the odd $L$--periodic function obtained via Lemma \ref{int}. 
In view of the identification of the cycle $\mathscr{C}$ with the interval $[0,kL]$, the function $\phi$ can be seen as a function 
on $\mathscr{C}$ that vanishes at the vertices of $\mathscr{C}$. Setting
\[
u(x)=\begin{cases}
\phi(x) & \text{if } x\in \mathscr{C}\\
0 & \text{otherwise on $\G$, }
\end{cases}
\]
we immediately see that $u\in H^1(\G)$, solves \eqref{NLS} on each edge of $\G$, satisfies the Kirchhoff conditions at every vertex of $\K$ and by construction is supported on $\mathscr{C}\subset \K$. Hence, letting
$  \mu=\int_{0}^{kL}|\phi|^2\dx$,
the proof is complete.
\end{proof}

\begin{remark}
The previous theorem applies also to the everywhere nonlinear problem, with the same proof.
\end{remark}

\begin{remark}
Theorem \ref{theocyc} goes beyond a mere breach of Theorem \ref{nonex} for graphs with cycles. It shows that the presence of a cycle (with some nice metric properties) immediately generates bound states supported on $\K$ (more precisely, on $\mathscr{C}$) \emph{for any value of} $\lambda\in\R$.
\end{remark}

\begin{remark}
 If $\G$ is a tadpole graph, such as that considered in \cite{CFN} and \cite{NPS}, then Theorem \ref{theocyc} holds without any restriction on the length of the cycle.
\end{remark}

\medskip
Exploiting the same technique, one can also exhibit trees with two or more pendants that admit bound states supported on $\K$ for any value of $\lambda\in\R$. This, in particular, shows that the assumption on the pendants in Theorem \ref{nonex} is necessary too.

In the next two examples we assume that $p>2$, $\lambda\in\R$ and $L>0$ are fixed and that $\phi$ is the odd $L$--periodic function provided by Lemma \ref{int}. Let $\overline{x}$ be a point in $(0,L)$ such that $\phi'(\overline{x})=\phi'(-\overline{x})=0$.

\begin{example}
The simplest noncompact tree having more than one pendant is the graph $\G$ of Figure \ref{twopen}. Assume (for simplicity) that the length of each pendant is $\overline{x}$, so that
 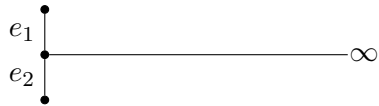
\begin{figure}[ht]
 \begin{center}
 \begin{tikzpicture}[xscale= 0.6,yscale=0.6]

 \node at (0,0) [nodo] (00) {};
 \node at (7,0) [infinito] (70) {$\infty$};
 \node at (0,1) [nodo] (01) {};
 \node at (0,-1) [nodo] (0-1) {};
 \node at (-.5,.5) [infinito] () {$e_1$};
 \node at (-.5,-.5) [infinito] () {$e_2$};

 \draw [-] (00) -- (70);
 \draw [-] (00) -- (01);
 \draw [-] (00) -- (0-1);
 \end{tikzpicture}
 \end{center}
 \caption{\footnotesize{a tree consisting of a half-line and two pendants.}}
 \label{twopen}
 \end{figure}
we can identify $e_1\cup e_2$ with the interval $[-\overline{x},\overline{x}]$. Hence, defining
\[
u(x)=\begin{cases}
\phi(x) & \text{if } x\in e_1\cup e_2\\
0 & \text{otherwise on $\G$, }
\end{cases}
\]
one sees that $u$ is a bound state of mass $\mu=\int_{-\overline{x}}^{\overline{x}}|\phi|^2\dx$, with Lagrange multiplier $\lambda$, supported on $\K= e_1\cup e_2$.
\end{example}

\begin{example}
 Let $\G$ be the graph with three pendants of Figure \ref{final}, with $\ell_{e_1}=\ell_{e_5}=\overline{x}$ and $\ell_{e_2}=\ell_{e_3}=\ell_{e_4}=L$. The path $e_1$--$e_5$ can be identified with the interval $[-\overline{x},3L+\overline{x}]$.

 \begin{figure}[ht]
 \begin{center}
 \begin{tikzpicture}[xscale= 0.6,yscale=0.6]

 \node at (-1,-1) [nodo] (-1-1) {};
 \node at (0,0) [nodo] (00) {};
 \node at (2,0) [nodo] (20) {};
 \node at (4,0) [nodo] (40) {};
 \node at (6,0) [nodo] (60) {};
 \node at (7,1) [nodo] (71) {};
 \node at (-1,5) [infinito] (-15) {$\infty$};
 \node at (1,-5) [infinito] (1-5) {$\infty$};
 \node at (3,5) [infinito] (35) {$\infty$};
 \node at (8,-5) [infinito] (8-5) {$\infty$};
 \node at (4,-2) [nodo] (4-2) {};
 \node at (-.8,-.2) [infinito] () {$e_1$};
 \node at (6.8,.2) [infinito] () {$e_5$};
 \node at (1,.3) [infinito] () {$e_2$};
 \node at (3,.3) [infinito] () {$e_3$};
 \node at (5,.3) [infinito] () {$e_4$};

 \draw [-] (-1-1) -- (00);
 \draw [-] (20) -- (00);
 \draw [-] (20) -- (40);
 \draw [-] (60) -- (40);
 \draw [-] (60) -- (71);
 \draw [-] (00) -- (-15);
 \draw [-] (00) -- (1-5);
 \draw [-] (20) -- (35);
 \draw [-] (40) -- (4-2);
 \draw [-] (60) -- (8-5);

 \end{tikzpicture}
 \end{center}
 \caption{\footnotesize{a tree with 4 half-lines and 3 pendants.}}
 \label{final}
 \end{figure}
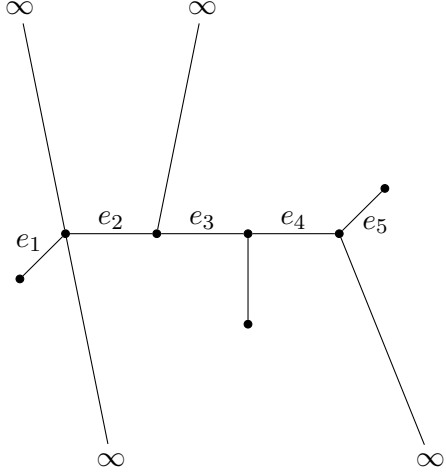
 
With the same function $\phi$ as above, defining
\[
u(x)=\begin{cases}
\phi(x) & \text{if } x\in \bigcup_{i=1}^5e_i\\
0 & \text{otherwise on $\G$, }
\end{cases}
\]
one sees that $u$ is a bound state with Lagrange multiplier $\lambda$, supported on $\bigcup_{i=1}^5e_i\subset \K$.
\end{example}

\section{Comparison with the everywhere nonlinear problem}
\label{full_non}

It is interesting, finally, to compare the phenomenon of nonexistence of bound states with small mass described in Theorem \ref{non_lapos}, to the 
case where the nonlinearity is present on the whole graph. Essentially, Theorem \ref{non_lapos} says that when $p\in[4,6)$ and $\mu$ is small there are neither ground states nor bound states supported on $\G$. In other words {\em all} solutions supported on $\G$ are ruled out for small values of the mass.

This feature is specific of problems with localized nonlinearity, since it has no analogue when the nonlinearity is present on the whole graph, as we now show. In particular, we  show that there are graphs such that \emph{for every value} of $\mu$ the everywhere nonlinear problem has {\em no} ground states, but {\em does admit} bound states supported on the whole graph.

We recall that a function $u$ is a bound state of mass $\mu$ for the everywhere nonlinear problem if and only if it satisfies Definition \ref{defbs} with \eqref{NLS} replaced by
\begin{equation}
 \label{NLSfull}
 u_e'' + |u_e|^{p-2}\,u_e = \lambda u_e.
\end{equation}

The associated NLS energy functional is
\[
\mathcal{E}(u,\G) = \frac12\int_\G |u'|^2\dx - \frac1p\int_\G |u|^p\dx,
\]
where we have made explicit the dependence on $\G$ for future use, 
and a ground state of mass $\mu$ is a minimizer of $\mathcal{E}(\,\cdot\,,\G)$ in $H_{\mu}^1(\G)$.

As an example of graph to illustrate the above discussion we consider the \emph{double bridge} graph $\B$ of Figure \ref{doubri}. 

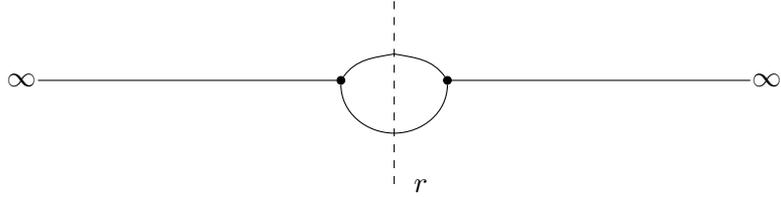
\begin{figure}[ht]
\begin{center}
\begin{tikzpicture}[xscale= 0.7,yscale=0.7]

\node at (-1,6) [nodo] (-16) {};
\node at (1,6) [nodo] (16) {};
\node at (-7,6) [infinito] (-76) {$\infty$};
\node at (7,6) [infinito] (76) {$\infty$};
\node at (.5,4) [infinito] () {$r$};

\draw [-] (-16) to [out=55,in=190] (0,6.5);
\draw [-] (0,6.5) to [out=-10,in=125] (16);
\draw [-] (-16) to [out=-90,in=-180] (0,5);
\draw [-] (0,5) to [out=0,in=-90] (16);
\draw [-] (16) -- (76);
\draw [-] (-16) -- (-76);
\draw [dashed] (0,7.5) -- (0,4);
\end{tikzpicture}
\end{center}
\caption{\footnotesize{a \emph{double bridge} graph, symmetric with respect to the axis $r$.}}
\label{doubri}
\end{figure}

By Theorem 1.2 of \cite{ast0} or Theorem 2.5 of \cite{ast1}, the graph $\B$ does not support ground states, for any value of $p\in(2,6)$ and $\mu>0$. However, the situation changes drastically when {\em bound} states are concerned.

\begin{proposition}
 \label{doubries}
 Let $p\in(2,6)$. For every $\mu>0$, there exists a bound state $u$ of mass $\mu$, symmetric with respect to the axis $r$ and everywhere positive on $\B$.
\end{proposition}

\begin{proof}
 Let $\widetilde{\B}$ be the graph on the left of the axis $r$, as in Figure \ref{Gtilde}.
 \begin{figure}[ht]
 \begin{center}
 \begin{tikzpicture}[xscale= 0.7,yscale=0.7]

 \node at (-1,6) [nodo] (-16) {};
 \node at (0,6.5) [nodo] (065) {};
 \node at (0,5) [nodo] (05) {};
 \node at (-7,6) [infinito] (-76) {$\infty$};
 \node at (.5,6.5) [infinito] (v1) {$\vv_1$};
 \node at (.5,5) [infinito] (v2) {$\vv_2$};

 \draw [-] (-16) to [out=55,in=190] (065);
 \draw [-] (-16) to [out=-90,in=-180] (05);
 \draw [-] (-16) -- (-76);
 \end{tikzpicture}
 \end{center}
 \caption{\footnotesize{the graph $\widetilde{\B}$.}}
 \label{Gtilde}
 \end{figure}
 
 By the results of \cite{ast2}, for every $\mu>0$ there exists $\widetilde{u}\in H_{\mu/2}^1(\widetilde{\B})$ such that
 \[
  \mathcal{E}(\widetilde{u},\widetilde{\B})=\inf_{v\in H_{\mu/2}^1(\widetilde{\B})}\mathcal{E}(v,\widetilde{\B}),
 \]
 that is, $\widetilde{u}$ is a ground state of mass $\mu/2$ on $\widetilde{\B}$ for the everywhere nonlinear problem. Notice that, since $\widetilde{u}$ is a minimizer of $\mathcal{E}$, up to multiplication by a constant phase, it can be taken real and everywhere positive (see again \cite{ast1}). Of course $\widetilde{u}$ solves \eqref{NLSfull} on the edges of $\widetilde{\B}$ and the Kirchhoff condition on the vertices. In particular, $\widetilde{u}'(\vv_1)=\widetilde{u}'(\vv_2)=0$.
 
 Now, consider the function $u:\B\to\R$, symmetric with respect to the axis $r$ and such that $u_{|\widetilde{\B}}=\widetilde{u}$. Obviously, $u\in H_{\mu}^1(\B)$, satisfies \eqref{kirchhoff} at the vertices of $\B$ and \eqref{NLSfull} on all the edges of $\B$, by the vanishing of $u'$ at $\vv_1$ and $\vv_2$. Thus $u$ is a bound state of mass $\mu$, symmetric on $\B$ and everywhere positive.
\end{proof}

Arguing as in the proof of Proposition \ref{doubries}, one can prove that an analogous result holds also for the \emph{triple bridge} graph.
\medskip

We conclude with a heuristic justification of this phenomenon. By Theorem  \ref{non_lapos}, there are no bound states supported on $\B$ (for the localized problem)
as soon as $ \ell \mu^{\frac{p-2}{6-p}}$ is small enough. If this were not the case, namely if we had a bound state for every $ \ell \mu^{\frac{p-2}{6-p}}$ small, we could construct a sequence $u_n$ of bound states of {\em fixed} mass each of them living on a double bridge graph $\B_n$, with $\ell_n \to 0$. Now, it is not difficult to show that this sequence of bound states tends
to a bound state on the limiting graph, which is $\R$ (the compact core, of length $\ell_n$, disappears in the limit). Thus we would have a nonzero $L^2$ solution of 
the {\em linear} problem $u'' = \lambda u$ on $\R$, which is impossible. This is where the presence of the localized nonlinearity is essential: if the compact core ``disappears'', we are left with a linear problem, that has no solution.

On the contrary, if we consider the everywhere nonlinear problem and we use the same argument, we end up in the limit with the problem $u'' + |u|^{p-1}u = \lambda u$ on $\R$, a problem  that {\em does have} solutions (the solitons). Thus, in this case there is no contradiction in the existence of bound states on all the $\B_n$'s.

\bigskip
\noindent\textbf{Acknowledgements}

\medskip
\noindent Lorenzo Tentarelli acknowledges the support of MIUR through the FIR grant 2013 ``Condensed Matter in Mathematical Physics (Cond-Math)'' (code RBFR 13WAET).


\begin{thebibliography}{99}

\bibitem{acfn3}
R. Adami, C. Cacciapuoti, D. Finco, D. Noja,
Fast solitons on star graphs.
\emph{Rev. Math. Phys.} {\bf 23} (2011), no. 4, 409--451. 

\bibitem{acfn}
R. Adami, C. Cacciapuoti, D. Finco, D. Noja,
Constrained energy minimization and orbital stability for the NLS equation on a star graph.
\emph{Ann. Inst. H. Poincar\'e Anal. Non Lin\'eaire} {\bf 31} (2014), no. 6, 1289--1310.

\bibitem{acfn2}
R. Adami, C. Cacciapuoti, D. Finco, D. Noja,
Variational properties and orbital stability of standing waves for NLS equation on a star graph.
\emph{J. Differential Equations} {\bf 257} (2014), no. 10, 3738--3777.

\bibitem{acfn4}
R. Adami, C. Cacciapuoti, D. Finco, D. Noja,
Stable standing waves for a NLS on star graphs as local minimizers of the constrained energy.
\emph{J. Differential Equations} {\bf 260} (2016), no. 10, 7397--7415.

\bibitem{ast0}
R. Adami, E. Serra, P. Tilli,
Lack of ground state for NLSE on bridge-type graphs. \emph{Mathematical technology of networks}, 1--11.
Springer Proc. Math. Stat., 128. Springer, Cham, 2015. 

\bibitem{ast1}
R. Adami, E. Serra, P. Tilli, 
NLS ground states on graphs.
\emph{Calc. Var. Partial Differential Equations} {\bf 54} (2015) no. 1, 743--761.

\bibitem{ast2}
R. Adami, E. Serra, P. Tilli, 
Threshold phenomena and existence results for NLS ground states on graphs.
\emph{J. Funct. Anal.} {\bf 271} (2016) no.1, 201--223.

\bibitem{berkolaiko}
G. Berkolaiko, P. Kuchment,
\emph{Introduction to quantum graphs}.
Mathematical Surveys and Monographs, 186. AMS, Providence, RI, 2013.

\bibitem{BL}
G. Berkolaiko, W. Liu,
Simplicity of eigenvalues and non--vanishing of eigenfunctions of a quantum graph.
http://arxiv.org/abs/1601.06225 (2016).

\bibitem{CFN}
C. Cacciapuoti, D. Finco, D. Noja,
Topology--induced bifurcations for the nonlinear Schr\"odinger equation
on the tadpole graph.
\emph{Phys. Rev. E (3)} {\bf 91} (2015), no. 1, article number 013206, 8pp.

\bibitem{cazenave}
T. Cazenave,
{\em Semilinear Schr\"odinger equations}.
Courant Lecture Notes in Mathematics, 10. AMS, Providence, RI, 2003.

\bibitem{gnusmi}
S. Gnutzmann, U. Smilansky,
Quantum graphs: Applications to quantum chaos and universal spectral statistics.
\emph{Adv. Phys.} {\bf 55} (2006), no. 5--6, 527--625.

\bibitem{der}
S. Gnutzmann, U. Smilansky, S. Derevyanko,
Stationary scattering from a nonlinear network.
\emph{Phys. Rev. A} {\bf 83} (2011), no. 3, article number 033831, 6pp.

\bibitem{dolbeault}
J. Dolbeault, M.J. Esteban, A. Laptev, M. Loss,
One--dimensional Gagliardo--Nirenberg--Sobolev inequalities: remarks on duality and flows. 
\emph{J. Lond. Math. Soc. (2)} {\bf 90} (2014), no. 2, 525--550. 

\bibitem{kottos}
T. Kottos, U. Smilansky,
Periodic orbit theory and spectral statistics for quantum graphs.
\emph{Ann. Physics} {\bf 274} (1999), no. 1, 76--124.

\bibitem{Kuchment}
P. Kuchment,
Quantum graphs. I. Some basic structures.
\emph{Waves Random Media} {\bf 14} (2004), no. 1, 107--128.

\bibitem{noja}
D. Noja,
Nonlinear Schr\"odinger equation on graphs: recent results and open problems.
\emph{Philos. Trans. R. Soc. Lond. Ser. A Math. Phys. Eng. Sci.} {\bf 372} (2014), no. 2007, 20130002, 20pp.

\bibitem{NPS}
D. Noja, D. Pelinovsky, G. Shaikhova,
Bifurcations and stability of standing waves in the nonlinear Schr\"odinger equation on the tadpole graph.
\emph{Nonlinearity} {\bf 28} (2015), no. 7, 2343--2378.

\bibitem{ST}
E. Serra, L. Tentarelli,
Bound states of the NLS equation on metric graphs with localized nonlinearities.
\emph{J. Differential Equations} {\bf 260} (2016), no. 7, 5627--5644.

\bibitem{T}
L. Tentarelli,
NLS ground states on metric graphs with localized nonlinearities.
\emph{J. Math. Anal. Appl.} {\bf 433} (2016), no. 1, 291--304.

\end{thebibliography}
\end{document}